\documentclass[11pt,a4paper]{article}

\usepackage{caption2}
\usepackage{CJK}
\usepackage{tabls}
\usepackage{bm}
\usepackage{multirow}
\usepackage{amssymb}
\usepackage{amsfonts}
\usepackage{mathrsfs}
\usepackage{empheq}
\usepackage{amsmath}
\usepackage{amsthm}
\usepackage{graphicx}
\usepackage{color}
\usepackage{indentfirst}
\usepackage{hyperref}
\usepackage{float}
\usepackage{cases}
\usepackage[titletoc]{appendix}

\allowdisplaybreaks%

\def\f{\frac}
\def\pa{\partial}

\def\e{\eqref}

\def\lab{\label}

\def\D{\Delta}

\def\R{\mathbb R}

\def\Z{\mathbb Z}
\def\C{\mathbb C}

\def \i{\mathrm i}

 \numberwithin{equation}{section}
\theoremstyle{definition}

 \newtheorem{thm}{Theorem}[section]
 \newtheorem{cor}{Corollary}[section]
 \newtheorem{lem}{Lemma}[section]
 \newtheorem{prop}{Proposition}[section]
 \theoremstyle{definition}
 \newtheorem{defn}{Definition}[section]
 \newtheorem{rem}{Remark}[section]

\newcommand{\be}{\begin{equation}}
\newcommand{\ee}{\end{equation}}
\newcommand{\beq}{\begin{equation*}}
\newcommand{\eeq}{\end{equation*}}

\begin{document}

\begin{CJK*}{GB}{gbsn}

\title{\bf  Asymptotic stability of wave equations coupled by velocities}

\author{ Yan Cui\thanks{School of Mathematical
Sciences, Fudan University, Shanghai 200433,
China. E-mail: \texttt{ycui14@fudan.edu.cn}. }
\quad Zhiqiang Wang\thanks{School of Mathematical Sciences and Shanghai Key Laboratory for Contemporary Applied Mathematics,
Fudan University, Shanghai 200433, P. R. China.
E-mail: \texttt{wzq@fudan.edu.cn}. This author was partially supported by the National Science Foundation of China (No. 11271082),
the State Key Program of National Natural Science Foundation of China (No. 11331004).}}

\date{June 3, 2015}

\maketitle

\begin{abstract}
This paper is devoted to study the asymptotic stability of wave equations  with constant coefficients coupled by velocities.
By using Riesz basis approach, multiplier method and frequency domain approach respectively,
we find the sufficient and necessary condition, that the  coefficients satisfy,
leading to the exponential stability  of the system. In addition, we give the optimal decay rate in one dimensional case.
 \end{abstract}



  \section{Introduction and Main Results}
  In this paper,  we consider the long time behavior of the solution to a system of wave equations with constant coefficients coupled by velocities.
  In particular, we  want to study what kind of conditions, that the coefficients satisfy, lead to the  exponential stability  of the system.

 In the case of scalar  wave equation,  there are numerous results on asymptotic stability or stabilization  with  internal or boundary damping.
 Cox  and Zuazua \cite{1994-CoxZZ} studied, by Fourier analysis, the energy decay  of
 \beq
 \lab{CZ}
 \begin{cases}
  u_{tt}-  u_{xx}+a(x)u_t =0 \quad & \text{in}\  (0,1)\times (0,+\infty), \\
  u(0,t)=u(1,t)=0   & \text{in} \  (0,+\infty),\\
  (u,u_{t})(0)=(u^0,u^1)  & \text{in} \ (0,1)
  \end{cases}
 \eeq
with indefinite damping.
As a corollary,  the exponential decay of  the energy and its optimal rate were shown in  \cite{1994-CoxZZ} when $a>0$  in $(0,1)$.
Using Multiplier method,
Alabau-Boussouira \cite{2002-AB-SICON, 2010-AB-JDE} and Alabau et al.  \cite{2002-ACK} proved the indirect stabilization of wave systems coupled by displacements,
for  instance,
  \beq
 \lab{fa}
 \begin{cases}
  u_{tt}- \Delta u+a(x)u_t+b(x) v =0  \quad & \text{in } \ \Omega\times (0,+\infty), \\
  v_{tt}- \Delta v-b(x) u=0  & \text{in } \ \Omega\times (0,+\infty),\\
  u=v=0   & \text{on} \ ~\Gamma\times  (0,+\infty),\\
  (u,u_{t})(0)=(u^0,u^1),  \  ( v,v_{t})(0)=(v^0,v^1)   & \text{in} \ ~ ~\Omega
 \end{cases}
 \eeq
That is, the damping is acted only in one equation and the total energy of the whole system decays polynomially due to the coupling effect.
 Using the criteria of polynomial decay in \cite{2005-LiuRao-ZAMP},
Liu  and Rao proved in \cite{2007-LiuRao-JMAA}, by frequency domain approach,
the polynomial stability of a partially damped wave system with  weak coupling by displacements and multiple propagation speeds.
 Liu  and Rao \cite{2009-LiuRao-DCDS} also proved the indirect stabilization, by Riesz basis approach,  with optimal polynomial decay rate for  a  wave system which is coupled by displacements in one dimensional case.

Recently, Alabau-Boussouira, Wang and Yu \cite{2015-AWY} obtained the indirect stabilization, combining multiplier method and weighted energy techniques,
  for  the following  damped wave system with variable coefficients coupled by velocities
 \beq
 \lab{awy}
  \left\{
 \begin{split}
  &u_{tt}- \Delta u+\rho(x ,u_{t})+b(x) v_{t} =0 \quad  &\text{in} \ \Omega\times (0,+\infty), \\
  &v_{tt}- \Delta v-b(x) u_{t}=0   &\text{in} \ \Omega\times (0,+\infty),\\
  &u=v=0   &\text{on} \ \Gamma\times  (0,+\infty), \\
  &(u,u_{t})(0)=(u^0,u^1),  \  ( v,v_{t})(0)=(v^0,v^1)  &\text{in}\  \Omega ~~~~~~\quad\qquad
  \end{split}
 \right.
 \eeq
 The decay speed is shown to  change corresponding to the various properties of the nonlinear damping $\rho(x,u_t)$,
 especially, if $\rho(x,u_t)\equiv \alpha u_t \ (\alpha>0),b(x)\equiv b>0 $, the total energy decays exponentially.
 This phenomenon indicates that the velocity coupling has different impact compared to the displacement coupling.
 Being regarded as perturbation, the coupling through displacements is  compact,  while the coupling through velocities is bounded.
On the other hand, the controllability (asymptotic stability) of coupled wave systems with general coefficients
 is closely related to the synchronization (asymptotic synchronization),
 according to the pioneer results on synchronization by Li and Rao \cite{2013-LR-CAM, 2014-LRH-COCV} (see also \cite{Hzy}).
 For the above  reasons, we focus on the question:
 What kind of coefficients can lead to the  exponential stability  of the general coupled wave system by velocities?

Suppose that $\Omega$ is an bounded open set in $\mathbb{R}^d$ with $C^2$ boundary $\Gamma=\pa \Omega$.
Consider the following general wave system coupled by velocities
 \be
 \lab{wave}
  \left\{
 \begin{split}
  &u_{tt}- \D u+\alpha u_{t}+\beta v_{t} =0 \quad\quad\   &\text{in } \ \Omega\times (0,+\infty),\\
 & v_{tt}- \D v+\gamma u_{t}+\eta v_{t}=0\quad\quad\quad &\text{in } \ \Omega\times (0,+\infty),\\
 & u=v=0\quad\quad\quad\quad  &\text{on} ~\  \Gamma\times  (0,+\infty),\\
 & (u,u_{t})(0)=(u^0,u^1),  \  ( v,v_{t})(0)=(v^0,v^1)\quad\quad &\text{in} \  ~\Omega\qquad \qquad\quad
  \end{split}
 \right.
 \ee
where $\alpha,\beta,\gamma, \eta \in \R$ are constants.

\begin{defn}
System \eqref{wave} is said to be exponential stable if there exist constants $M>0$ and $\omega >0$ such that
\be \lab{en1} E(t)\leq Me^{-\omega t}E(0) \ee
where the (total) energy is defined by
\be \lab{wave-energy}
  E(t)=\f12\int_\Omega|u_t|^2+|\nabla u|^2+|v_{t}|^2+|\nabla v|^2dx \ee
and $\omega$ is the corresponding decay rate.
\end{defn}

The task of this paper is to find the conditions that the coefficients matrix
$\begin{pmatrix}  \alpha & \beta \\
                     \gamma & \eta \\  \end{pmatrix}$
 should satisfy such that System \eqref{wave} is exponential stable.
Our main result is the following theorem:
 \begin{thm}
\lab{thm1}
System \eqref{wave} is exponential stable if and only if
\be
 \lab{2}
 \begin{cases}
 \alpha+\eta>0,\\

                 \alpha \eta-\beta\gamma>0 \end{cases}
                 \ee
i.e.,  the two eigenvalues of the coefficient matrix
$\begin{pmatrix}  \alpha & \beta \\   \gamma & \eta \\  \end{pmatrix}$ both have positive real part.
\end{thm}

\begin{rem}
By Proposition \ref{prop2}, the equivalent form of Theorem \ref{thm1}, one can see that
the condition \eqref{2} means that the two components of the solution are essentially both damped.
\end{rem}

The organization of the paper is as follows:
In Section 2, we give the well-posedness  of the system \eqref{wave} and reduce the original general problem equivalently into the same problem in two canonical forms
(see  Proposition \ref{prop2}).  Then, we prove  Proposition \ref{prop2} in Section 3
by three different approaches successively, that is,  Riesz basis approach, multiplier method, frequency domain approach.
Finally,  some useful extension and remarks are provided in Section 4.

 \section{Preliminaries}

 For the simplicity of statements, let
 $\mathcal{L}^2=\mathcal{L}^2(\Omega),\mathcal{H}^1_0=\mathcal{H}^1_0(\Omega),\mathcal{H}^2=\mathcal{H}^2(\Omega)$,
 and
 $\mathcal{H}=\mathcal{H}^1_0\times \mathcal{L}^2\times \mathcal{H}^1_0\times \mathcal{L}^2$.  Then
$\mathcal{H}$ is a complex Hilbert space equipped with the inner product
\be
\lab{inpro}
\langle U,\tilde{U} \rangle_\mathcal{H}=\int_\Omega \nabla  \overline{u} \cdot \nabla  \tilde{u}
+ \overline{y} \tilde{y} +\nabla \overline{v} \cdot \nabla \tilde{v}  + \overline{z} \tilde{z} \, dx.
\ee

Let $U=(u,y,v,z)$ and let $U_0=(u_0,u_1,v_0,v_1)$.
System \eqref{wave} can be rewritten as the Cauchy Problem in abstract form:
\be \lab{abseq}
  \f{dU}{dt}=\mathcal{A}U,  \quad U(0)=U_0.\ee
where $\mathcal{A}:\mathcal{D}(\mathcal{A})\rightarrow \mathcal{H}$ is defined by
 \be   \lab{3}
     \mathcal{A}U=(y,\D u-\alpha y-\beta z,z,\D v -\gamma y -\eta z)        \ee
 with $ \mathcal{D}(\mathcal{A}) = (\mathcal{H}^2\cap \mathcal{H}^1_0)\times \mathcal{H}^1_0\times
     (\mathcal{H}^2\cap \mathcal{H}^1_0)\times \mathcal{H}^1_0. $
 Obviously, $\mathcal{A}$ can be regarded as a bounded perturbation to the standard wave operator,
 namely, $\mathcal{A}=\mathcal{A}_0+\mathcal{A}_1$
 where  $\mathcal{A}_0U=(y,\D u,z,\D v )$ with $\mathcal{D}(\mathcal{A}_0)=\mathcal{D}(\mathcal{A})$
 and $ \mathcal{A}_1U=(0,-\alpha y-\beta z,0,-\gamma y-\eta z)$ with  $\mathcal{D}(\mathcal{A}_1)=\mathcal{H}$.

 By classical Hille-Yoshida Theorem and perturbation theory (see \cite[page 76, Theorem 1.1]{Pazy}), we have
 \begin{prop} \lab{prop1}
 $\mathcal{A}$ generates a $C_0$-semigroup $\{S(t)= e^{t \mathcal{A}}\}_{t \geq 0}$ on $\mathcal{H}$.
 If  $U_0=(u_0,u_1,v_0,v_1)\in \mathcal{H},$
 System \eqref{wave} has a unique solution $U(t) =S(t) U_0 \in C^0([0,+\infty);\mathcal{H}).$
If  $U_0 \in \mathcal{D}(\mathcal{A}),$
System  \eqref{wave} has a unique solution $U(t) =S(t) U_0 \in  C^0([0,+\infty); \mathcal{D}(\mathcal{A}))\cap C^1([0,+\infty);\mathcal{H}) .$
Moreover, $\|U(t)\|_{\mathcal{H}}^2 = 2E(t)$ for all $t\geq 0$.
\end{prop}

For the sake of convenience of statement, we use real Schur decomposition for the coefficient  matrix
so that we can reduce the original general problem into the same problem in two canonical forms.

\begin{lem} \cite[page 79, Theorem 2.3.1]{GolubBook}
 \lab{lem1}
For any $B=\begin{pmatrix}  \alpha & \beta \\   \gamma & \eta \\  \end{pmatrix} \in \mathbb{R}^{2\times2} $,
there exists real orthogonal matrix $P \in  \mathbb{R}^{2\times2} $  such that
  $P^TBP=\tilde{B}$ where $\tilde{B}$ is in one of the following two canonical forms: \\
 i) $  \tilde{B}=\begin{pmatrix}  a & b \\
                                 -b & a \\ \end{pmatrix} $
         $\ (a,b \in \mathbb{R}, b\neq 0 ) $£» \\
 ii)  $ \tilde{B}=\begin{pmatrix}   a & 0 \\
                                                      b & c \\   \end{pmatrix} $
                                                      $\  (a,b,c \in \mathbb{R}) $. \\
  Moreover,  \eqref{2} is equivalent to $a>0$ in Case i) and   to $a>0$ and $c>0$ in Case ii).
 \end{lem}

 \begin{lem}
\lab{lem2}
Let $P \in \mathbb{R}^{2 \times 2}$ be a real orthogonal matrix and let $(u,v)$ be solution of  \eqref{wave},
then $(\tilde{u},\tilde{v})=(u,v)P$ is solution of
\be
 \lab{wave1}
  \left\{
 \begin{split}
  &\tilde{u}_{tt}- \D  \tilde{u}+\tilde{\alpha} \, \tilde{u}_{t}+\tilde{\beta} \, \tilde{v}_{t} =0 \quad\quad\  &\text{in } \ \Omega\times (0,+\infty),\\
  &\tilde{v}_{tt}- \D \tilde{v}+ \tilde{\gamma} \, \tilde{u}_{t}+\tilde{\eta} \, \tilde{v}_{t}=0\quad\quad\quad &\text{in } \ \Omega\times (0,+\infty),\\
  &\tilde{u}=\tilde{v}=0\quad\quad\quad\quad&\text{on} ~ \  \Gamma\times  (0,+\infty),\\
  &(\tilde{u},\tilde{v})(0)=(u^0,v^0)P , \quad (\tilde{u}_t,\tilde{v}_t)(0)=(u^1,v^1)P  \quad\quad&\text{in} ~\ \Omega \qquad\qquad\quad
  \end{split}
 \right.
 \ee
where
   $\tilde{B} = \begin{pmatrix}    \tilde{\alpha} & \tilde{\beta} \\
                                                               \tilde{\gamma} & \tilde{\eta} \\   \end{pmatrix}
                            = P^TBP$
       and the energy of    \eqref{wave1}
is equivalent to that of \eqref{wave}.
 \end{lem}

Thanks to Lemmas \ref{lem1}-\ref{lem2}, Theorem \ref{thm1} is equivalent to the following Proposition
\begin{prop}
\lab{prop2}
System \eqref{wave} is exponential stable if and only if  \\
  i)   $a>0$  when $B= \begin{pmatrix}
      a & b \\
      -b & a \\ \end{pmatrix} \ (a,b \in \mathbb{R}, b\neq 0 ) $;   \\
  ii) $a>0$ and $c>0$ when  $B= \begin{pmatrix}      a & 0 \\
      b & c \\ \end{pmatrix} \  (a,b,c \in \mathbb{R}) $.
\end{prop}

\section{Proof of  Proposition \ref{prop2}  }

\subsection{Riesz basis approach}
In this subsection,  we adopt Riesz basis approach to prove Proposition \ref{prop2} in one dimensional case, that is $\Omega= (0,\pi) \subset \R$.
The key ingredient is to prove the Riesz basis property and to analyze the spectrum of  $\mathcal{A}$.
Additionally, we obtain the optimal and explicit  decay rate in this situation, see Corollary \ref{cor}.

In Case  i),
System \eqref{wave} can be rewritten as the Cauchy Problem  \eqref{abseq} in one dimensional case,
where $\mathcal{A}:\mathcal{D}(\mathcal{A})\rightarrow \mathcal{H}$ is defined by
 \be   \lab{4}
    \mathcal{A}U=(y, u_{xx}-a y-b z,z, v_{xx} +b y -a z)       \ee
 with $ \mathcal{D}(\mathcal{A}) = (\mathcal{H}^2\cap \mathcal{H}^1_0)\times \mathcal{H}^1_0\times
     (\mathcal{H}^2\cap \mathcal{H}^1_0)\times \mathcal{H}^1_0.$

Let $\lambda$ be the eigenvalue of  $\mathcal{A}$ and $E=(u,y,v,z) \in \mathcal{D}(\mathcal{A})$ be its eigenvector:
   \beq\lab{suanzifangcheng}
      (\lambda I-\mathcal{A})E=0\eeq
It is equivalent to
\beq\left\{ \begin{split}
         & \lambda u-y = 0  \quad&\text{in} ~&\mathcal{H}_0^1 \\
           &(\lambda+a)y- u_{xx}+b z= 0  \quad&\text{in} ~&\mathcal{L}^2\\
           &\lambda v-z= 0    \quad&\text{in}~ &\mathcal{H}_0^1\\
          & (a+\lambda)z- v_{xx}-b y= 0 \quad&\text{in} ~&\mathcal{L}^2
        \end{split}
\right.\eeq
or further
\beq\left\{ \begin{split}
        &- u_{xx}+(\lambda+a)\lambda u+b\lambda v=0   \quad&\text{in} ~&\mathcal{L}^2  \\
          & - v_{xx}+(\lambda+a)\lambda v-b\lambda u=0   \quad&\text{in} ~&\mathcal{L}^2 \\
          & u(0)=u(\pi)=v(0)=v(\pi)=0 &
        \end{split}
\right.\eeq
Canceling $v$, we obtain a forth order ODE of $u$:
\be
\lab{sijie}
\begin{cases}
         u_{xxxx}-2\lambda(\lambda+a)u_{xx}+[(a+\lambda)^2+b^2]\lambda^2u=0 \\
           u(0)=u(\pi)=u_{xx}(0)=u_{xx}(\pi)=0
        \end{cases}
  \ee
whose  general solution is given by $u(x)=A_1e^{\nu_1x}+A_2e^{\nu_2x}+A_3e^{\nu_3x}+A_4e^{\nu_4x}$,
where $\nu_i \ (i=1,\cdots, 4)$  are four zeroes of the algebraic equation
  \be \nu^4-2(\lambda+a)\lambda\nu^2+[(a+\lambda)^2+b^2]\lambda^2=0
   \ee
and  $A_i\ (i=1,\cdots, 4)$ are constants to be determined.
 Then it follows that
 \beq  u(x)=\f{ \sin nx  }{n}, \ \nu=\pm  \i n   \quad (n\in \Z^+)
  \eeq
and consequently
 \beq v(x)= \frac{\lambda^2+a \lambda +n^2}{-b\lambda} u(x),  \ y(x)=\lambda u(x), \ z(x)=\lambda v(x) \eeq
 where the eigenvalue $\lambda$ satisfies the characteristic equation
\beq \lab{sp}
   \lambda^4+ 2a \lambda^3+(2n^2+a^2+b^2)\lambda^2+2a n^2\lambda+n^4=0 \eeq
 or equivalently
 \beq
  (\lambda^2+(a- \i b)\lambda+n^2)(\lambda^2+(a+\i b)\lambda+n^2)=0 \eeq
Consequently, there are two classes of eigenvalues:
\be
\lab{biao1}
          \lambda_{1, n}^\pm=(\pm X_n -\f{a}{2}) - \i(\pm Y_n -\f{b}{2}), \quad
     \lambda_{2, n}^\pm=(\pm X_n-\f{a}{2})+ \i ( \pm Y_n - \f{b}{2})
      \ee
where
  \be \lab{XnYn} X_n =\sqrt{\f{\sqrt{(a^2-b^2-4n^2)^2+4a^2b^2}+a^2-b^2-4n^2}{2}},
                \quad Y_n =\f{ab}{X_n} \ee
 yielding that  $X_n$ is a decreasing function of  $n\in \Z ^+$.
It follows that
\begin{align}
 \lab{lam1}
&  \max  \limits_{n\in \Z^+}  \{ \Re(\lambda_{1, n}^\pm),  \Re (\lambda_{2, n}^\pm ) \} =  \Re(\lambda_{1, 1}^+) = \Re(\lambda_{2, 1}^+) =X_1 -\f{a}{2}
 \\
\lab{jianjin}
 & \lambda_{1, n}^\pm \sim \pm \i n,\  \lambda_{2, n}^\pm \sim \pm \i n, \quad \text{\ as\ } n\rightarrow +\infty
\end{align}
The corresponding eigenvectors are
\beq E_{1, n}^\pm=\sin nx  (\f{1}{n},\f{\lambda_{1,n}^\pm}{n},\f{-\i}{n},\f{- \i \lambda_{1,n}^\pm}{n} ) ,\quad
 E_{2, n}^\pm=\sin nx   (\f{1}{n},\f{\lambda_{2,n}^\pm}{n},\f{\i}{n},\f{\i \lambda_{2,n}^\pm}{n} )  \eeq

If $a \leq 0$, there exists an eigenvalue with nonnegative real part  according to  \eqref{biao1}-\eqref{XnYn}.
Then by choosing the corresponding eigenvector as the initial data,
  it is easy to conclude that  System \eqref{wave} is unstable.  If $a>0$, we have
\begin{prop} \label{prop3}
For any given $a>0,b \in \R,b\neq0 $, $\{ (E_{1,n}^+,E_{1,n}^-,E_{2,n}^+,E_{2,n}^-) \} _{n\in \Z^+}$ forms a Riesz basis of the Hilbert space $\mathcal{H}$.
\end{prop}

\begin{proof}
Since $a>0,b\neq0$,   \eqref{wave} has no multiple eigenvalues, thus $\{ (E_{1,n}^+,E_{1,n}^-,E_{2,n}^+,E_{2,n}^-) \} _{n\in \Z^+}$ are linearly independent.
Note that
  \beq  e_{1,n}^\pm=\sin nx \, (\f{1}{n},\pm \i,-\f{\i}{n},\pm 1) ,\quad
        e_{2,n}^\pm=\sin nx  \, (\f{1}{n},\pm \i,\f{\i}{n},\mp 1) \quad (n\in \Z^+)
        \eeq
forms  a Risez basis   of $\mathcal{H}$ and
\begin{align*}
&  \langle e_{i,n}^+,e_{j,m}^+ \rangle_{\mathcal{H}}= \langle e_{i,n}^+,e_{j,m}^-\rangle_{\mathcal{H}}= \langle e_{i,n}^-,e_{j,m}^-\rangle_{\mathcal{H}}=0,
     \quad \forall n, m\in \mathbb{Z}^+,n\neq m,i,j=1,2
     \\
& \langle e_{1,n}^+,e_{2,n}^+ \rangle_{\mathcal{H}}=\langle e_{1,n}^+,e_{2,n}^- \rangle_{\mathcal{H}}
     =\langle e_{1,n}^-,e_{2,n}^+ \rangle_{\mathcal{H}}=\langle e_{1,n}^-,e_{2,n}^-\rangle_{\mathcal{H}}=0,
    \quad \forall n\in \mathbb{Z}^+
\end{align*}
Then using the asymptotic expansion of the eigenvalues \eqref{jianjin}, we get
 \beq   || e_{i,n}^+-E_{i,n}^+||_{\mathcal{H}}^2+|| e_{i,n}^--E_{i,n}^-||_{\mathcal{H}}^2 = O_{\mathcal{H}}(\f{1}{n^2})  \quad  i=1,2 \eeq
  and thus
   \beq \sum \limits_{n\in\Z^+} \mathop{\sum}\limits_{i={1,2}}
      \big( || e_{i,n}^+-E_{i,n}^+||_{\mathcal{H}}^2+|| e_{i,n}^--E_{i,n}^-||_{\mathcal{H}}^2 \big) <+\infty  \eeq
By Lemma  \ref{Riesz},
$\{ (E_{1,n}^+,E_{1,n}^-,E_{2,n}^+,E_{2,n}^-) \} _{n\in \Z^+}$ forms a Riesz basis of $\mathcal{H}$.
\end{proof}

Thanks to Proposition \ref{prop3}, for any given $U_0\in \mathcal{H}$,
there exist  $\{(\alpha_{1,n}^+,\alpha_{1,n}^-,\alpha_{2,n}^+,\alpha_{2,n}^-) \}_{n \in \Z^+}$
$\subset \C^4$ such that
  \be \lab{inidata} U_0=\mathop{\sum}_{n\in \Z^+}\mathop{\sum}_{i=1,2}  \big( \alpha_{i,n}^+E_{i,n}^+ +\alpha_{i,n}^-E_{i,n}^- \big)  \ee
and
 \be \lab{con1}
 E(0) =2 \|U_0\|^2_{\mathcal{H}}  \sim  \mathop{\sum}_{n\in \Z^+} \mathop{\sum}_{i=1,2}
  \big(  \big|\alpha_{i,n}^+\big|^2+ \big|\alpha_{i,n}^-\big|^2 \big) <+\infty
   \ee
Moreover, the solution of  \eqref{abseq} writes
   \be \lab{sol1}
     U(t)=\mathop{\sum}_{n\in \Z^+}\mathop{\sum}_{i=1,2}
     \big( \alpha_{i,n}^+  e^{\lambda_{i,n}^+t} E_{i,n}^+ +\alpha_{i,n}^- e^{\lambda_{i,n}^-t} E_{i,n}^- \big)  \ee

%

%

 In Case ii),  consider the eigenvalue and eigenfunction:
 $(\lambda I-\mathcal{A})E=0,\  E\in \mathcal{D}(\mathcal{A})$, namely
    \beq (y, u_{xx}-ay,z, v_{xx}-by-cz)=\lambda(u,y,v,z)  \in \mathcal{H} \eeq
 Then the eigenvalue satisfies
  \beq   \lambda^4+(a+c)\lambda^3+(2n^2+ac)\lambda^2+(a+c)n^2\lambda+n^4=0   \eeq
or equivalently,
   \beq (\lambda^2+a\lambda+n^2)(\lambda^2+c\lambda+n^2)=0 \eeq
thus
\be
\lab{biao3}
\lambda_{1, n}^\pm=
\begin{cases}
    \displaystyle       \frac{-a \pm \sqrt{a^2-4n^2}}{2},   \text{\ if \ }n\leq\f{a}{2}, \\
    \displaystyle          \frac{-a  \pm \i \sqrt{4n^2-a^2}}{2},    \text{\ if \ } n>\f{a}{2},  \\
\end{cases}
\lambda_{2, n}^\pm=
\begin{cases}
    \displaystyle    \frac{-c  \pm \sqrt{c^2-4n^2}}{2},    \text{\ if \ }n\leq\f{c}{2}, \\
   \displaystyle    \frac{-c \pm \i \sqrt{4n^2-c^2}}{2},   \text{\ if \ } n>\f{c}{2} ,
\end{cases}
    \ee
Clearly
\begin{align}
 \lab{lam1-2}
&  \max  \limits_{n\in \Z^+}   \Re(\lambda_{i, n}^\pm)  =  \Re(\lambda_{i, 1}^+)  \quad i=1,2
 \\
\lab{jianjin2}
 & \lambda_{1, n}^\pm \sim \pm \i n,\  \lambda_{2, n}^\pm \sim \pm \i n, \quad \text{\ as\ } n\rightarrow +\infty
\end{align}
Therefore all the eigenvalues
$\{\lambda_{1,n}^\pm,\lambda_{2,n} ^\pm\}_{n\in \Z^+}$ have negative real parts if and only if
$a>0$ and $c>0$. Certainly if $a \leq 0$ or $c \leq 0$, System \eqref{wave} is not asymptotically stable.

When $a>0$ and $c>0$, we would like to prove that the eigenvectors as well as root-vectors form a Riesz basis of $\mathcal{H}$.
For this purpose, we discuss the different situations that multiple eigenvalues may appear with various values of $a$ and $c$.

$\bullet$ Case 1. $a \not \in 2\mathbb{Z}^+, c \not \in 2\mathbb{Z}^+$ and $a \neq c$.
In this case,$\lambda_{1,n}^+$,$\lambda_{1,n}^-$,$\lambda_{2,n}^+$,$\lambda_{2,n}^-$ are distinct and  System \eqref{wave} is decoupled.
The eigenvectors corresponding to $\lambda_{1,n}^\pm,\lambda_{2,n}^\pm \ (n\in \Z^+)$ are
 \beq  E_{1,n}^\pm=\sin nx \, (\f{1}{n},\f{\lambda_{1,n}^\pm }{n} ,\f{b}{(a-c)n},\f{\lambda_{1,n}^\pm b}{(a-c)n}), \quad
 E_{2,n}^\pm=\sin nx  \, (0,0,\f{1}{n},\f{\lambda_{2,n}^\pm}{n})
 \eeq
Comparing $\{E_{1,n}^\pm, E_{1,n}^\pm\}$ with  a Riesz basis of $\mathcal{H}$:
 \beq  e_{1, n}^\pm=\sin nx \, (\f{1}{n},\pm \i, \f{b}{(a-c) n},  \f{ \pm \i b}{a-c} ) ,
  \quad e_{2, n}^\pm=\sin nx \, (0,0,\f{1}{n}, \pm \i )
  \eeq
yields by \eqref{jianjin2} and Lemma \ref{Riesz} that
$\{E_{1,n}^+,E_{1,n}^-,E_{2,n}^+,E_{2,n}^-\}_{n \in \Z^+ }$ forms a Riesz basis of $\mathcal{H}$.
Therefore, every initial data $U_0\in \mathcal{H}$ can be expanded as  \eqref{inidata} with \eqref{con1} and the corresponding solution is given by \eqref{sol1}.

 $\bullet$ Case 2.   $a \in 2\mathbb{Z}^+$ or  $c \in 2\mathbb{Z}^+$ and $a\neq c$.  There are three sub-cases.

 $\circ$ Case 2.1.  $a=2m \neq c=2k,\ m,k\in \Z^+$.   System \eqref{wave} has two multiple  eigenvalues:
 $\lambda_{1,m}^\pm =-\f{a}{2}=-m$  and $\lambda_{2,k}^\pm=-\f{c}{2}=-k$.
 The dimension of the eigen space corresponding to  $\lambda_{1,m}^\pm$ or $\lambda_{2,k}^\pm$ is one.
The corresponding eigen function corresponding to $\lambda_{1,m}^\pm$ is
              \beq E_{1,m}^+=\sin mx\, (\f{1}{m},-1,\f{b}{(a-c)m},\f{- b}{a-c}) \eeq
Solving $(\mathcal{A}-\lambda_{1, m}^+I)E_{1, m}^-=E_{1, m}^+$, we obtain the root vector
            \beq E_{1,m}^-=\f{\sin mx}{2m} \, (\f{1}{m},1,\f{b}{(a-c)m},\f{b}{a-c})  \eeq
which is linearly independent of   $E_{1,m}^{+}$.  Similarly , we calculate the eigenvector and root vector of  $\lambda_{2,k}^+=-k$:
\beq E_{2,k}^+=\sin kx\, (0,0,\f{1}{k},-1), \quad
              E_{2,k}^-=\f{\sin kx}{2k}\, (0,0,\f{1}{k},1) \eeq

By Lemma \ref{Riesz}, $\{E_{1,n}^+,E_{1,n}^-,E_{2,n}^+,E_{2,n}^-\}_{n \in \Z^+ }$ forms a Risez basis of $\mathcal{H}$.
For every $U_0\in \mathcal{H}$ given by \eqref{inidata} with \eqref{con1},
the solution of \eqref{abseq} is
\be\lab{sol2}
\begin{split}U(t)= \mathop{\sum}_{n\in \Z^+} \mathop{\sum}\limits_{i={1,2}}    \big(\alpha_{i,n}^+ e^{\lambda_{i,n}^+t} E_{i,n}^+
   +\alpha_{i,n}^-  e^{\lambda_{i,n}^-t} E_{i,n}^- \big)
   + \alpha_{1,m}^-  t e^{\lambda_{1,m}^+ t} E_{1,m}^+ + \alpha_{2,k} ^- t e^{\lambda_{2,k}^-t} E_{2,k}^+
\end{split}
\ee

 $\circ$ Case 2.2.   $a=2m, m\in\Z^+$ and $c \not\in2 \Z^+$.
 $\mathcal{A}$ has one multiple eigenvalue $\lambda_{1,m}^\pm =-\f{a}{2}=-m$.
 For every $U_0\in \mathcal{H}$ given by \eqref{inidata} with \eqref{con1},
the solution of \eqref{abseq} is
 \be\lab{sol3}
\begin{split}U(t)= \mathop{\sum}_{n\in \Z^+}\mathop{\sum}_{i=1,2}
  \big( \alpha_{i,n}^+ e^{\lambda_{i,n}^+t}E_{i,n}^+ +\alpha_{i,n}^-  e^{\lambda_{i,n}^-t} E_{i,n}^- \big)
  +\alpha_{1,m}^- t  e^{\lambda_{1,m}^+t } E_{1,m}^+
\end{split}
\ee

$\circ$ Case 2.3.  $c=2k,k\in\Z^+$ and $a\not\in2 \Z^+$.
$\mathcal{A}$ has one multiple eigenvalue $\lambda_{2,k}^\pm= -\f{c}{2}=-k$.
Similar to Case 2.2,
 for every $U_0\in \mathcal{H}$ given by \eqref{inidata} with \eqref{con1},
the solution of \eqref{abseq} is
\be \lab{sol4}
\begin{split}U(t)=\mathop{\sum}_{n\in \Z^+}\mathop{\sum}_{i=1,2}
\big( \alpha_{i,n}^+ e^{\lambda_{i,n}^+t} E_{i,n}^+ +\alpha_{i,n}^- e^{\lambda_{i,n}^-t} E_{i,n}^- \big)
  +\alpha_{2,k}^-  t e^{\lambda_{2,k}^+t} E_{2,k}^+
\end{split}
\ee

$\bullet$ Case   3.   $a=c \not \in 2\mathbb{Z}^+$.
The two classes of eigenvalues coincide $\lambda_{1,n}^+=\lambda_{2,n}^+$, $\lambda_{1,n}^-=\lambda_{2,n}^-$,
and the dimension of eigenspace   varies with respect to $b$.

$\circ$  Case 3.1.  $b=0$. System \eqref{wave} is decoupled.
  The dimension of eigenspace of $\lambda_{1,n}^{\pm}$ is two.
The eigenvectors are
    \beq   E_{1, n}^\pm= \sin nx \, (\f{1}{n},\f{\lambda_{1, n}^\pm}{n},0,0),
     \quad E_{2, n}^\pm= \sin nx \, (0,0,\f{1}{n},\f{\lambda_{1, n}^\pm}{n} )
     \eeq
  which forms a Riesz basis of $\mathcal{H}$.
For every $U_0\in \mathcal{H}$ given by \eqref{inidata} with \eqref{con1},
the solution of \eqref{abseq} is still \eqref{sol1}.

$\circ$ Case 3.2. $b\neq 0$.
 The dimension of eigenspase of $\lambda_{1,n}^+, \lambda_{1,n}^-$ is one.
 The eigenvector and root vector are
 \begin{align*}
  & E_{1,n}^\pm=\sin nx \, (0,0,\f{1}{n},\f{\lambda_{1,n}^\pm}{n}),
   &  E_{2, n}^\pm= - \f{\sin nx}{n} \, (\f{a+ 2\lambda_{1, n}^\pm}{b \lambda_{1, n}^\pm},\f{a+2\lambda_{1, n}^\pm}{b},\f{1}{ 2\lambda_{1, n}^\pm},-\f{1}{2})
              \end{align*}
  which forms a Riesz basis of $\mathcal{H}$.
For every $U_0\in \mathcal{H}$ given by \eqref{inidata} with \eqref{con1},
the solution of \eqref{abseq} is
    \be\lab{sol5}
  \begin{split}U(t)=\mathop{\sum}_{n\in \Z^+} \Big(\mathop{\sum}_{i=1,2}
    \alpha_{i,n}^+  e^{\lambda_{i,n}^+t} E_{i,n}^+ +\alpha_{i,n}^- e^{\lambda_{i,n}^-t} E_{i,n}^-
     + \alpha_{2,n}^+ t e^{\lambda_{1,n}^+t} E_{1,n}^+ + \alpha_{2,n}^- te^{\lambda_{1,n}^-t} E_{1,n}^-  \Big)
  \end{split}
  \ee

 $\bullet$ Case  4.  $a=c=2m, m\in 2\mathbb{Z}$.  Then $\lambda_{1, n}^+=\lambda_{2, n}^+,  \lambda_{1, n}^-=\lambda_{2, n}^-$ for all $n\in \Z^+$.

 $\circ$ Case 4.1 $b= 0$.
 System \eqref{wave} is decoupled.  For $n\neq m$,  the algebric degree of the eigenvalue  $\lambda_{1,n}^\pm$ is two.
 The dimension of eigenspase of $\lambda_{1,n}^{\pm}$ is two and the eigenvectors are
   \beq \label{34} E_{1, n}^\pm= \sin nx \, (\f{1}{n},\f{\lambda_{1, n}^\pm}{n},0,0),\quad
   E_{2, n}^\pm= \sin nx\,  (0,0,\f{1}{n},\f{\lambda_{2, n}^\pm}{n} )
   \eeq
For $n=m$,  the algebric degree of the eigenvalue  $\lambda_{1,m}^+$ is four. Actually $\lambda_{1, m}^\pm=\lambda_{2,m }^\pm=-m$,
and the  dimension of its eigenspase is two.
 The eigenvectors and root eigenvectors are
 \begin{align*}
   &E_{1,m}^+=\sin mx\, (\f{1}{m},-1,0,0),\quad
   E_{2,m}^+=\sin mx\, (0,0,\f{1}{m},-1)
    \\
  & E_{1, m}^-=\f{\sin mx}{2m} \,  (\f{1}{m},1,0,0),
  \quad E_{2, m}^-=\f{\sin mx}{2m} \, (0,0,\f{1}{m},1)
  \end{align*}
For every $U_0\in \mathcal{H}$ given by \eqref{inidata} with \eqref{con1},
the solution of \eqref{abseq} is
  \be \lab{sol6}
\begin{split}U(t)=\mathop{\sum}_{n\in \Z^+}\mathop{\sum}_{i=1,2}
\big( \alpha_{i,n}^+  e^{\lambda_{i,n}^+t} E_{i,n}^+ +\alpha_{i,n}^- e^{\lambda_{i,n}^-t} E_{i,n}^- \big)
+ \alpha_{2,m}^ +  t e^{\lambda_{1,m}^+t} E_{1,m}^+ + \alpha_{2,m}^- t e^{\lambda_{1,m}^-t} E_{2,m}^-
\end{split}
\ee

  $\circ$   Case 4.2  $b\neq 0$.   For $n\neq m$,  the eigenvalues  $\lambda_{1,n}^\pm$ are the same as in Case 3.2.
 While for $n=m$, the dimension of the eigenspace  of the eigenvalue $\lambda_{1,m}^\pm=\lambda_{2,m}^\pm =-m$ is one.
 We could solve successively the eigenvalue, first order, second order and the third order  root vector:
 \begin{align*}
   & E_{1,m}^+=\sin mx \, (0,0,\f{1}{m},-1), \quad
   E_{1,m}^-=\f{\sin mx}{2m} \, (0,0,\f{1}{m},1), \\
  &    E_{2, m}^+=\f{\sin mx}{m} \, (\f{1}{b m},\f{-1}{b},\f{1}{2 m^2},0), \quad
    E_{2,m}^-=\f{\sin mx}{m^2} \, (\f{3}{2b m},\f{-1}{2b},\f{1}{2m^2},0).
\end{align*}
For every $U_0\in \mathcal{H}$ given by \eqref{inidata} with \eqref{con1},
the solution of \eqref{abseq} is
  \be \lab{sol7}
   \begin{split}
   U(t)&=\mathop{\sum}_{n\in \Z^+}\mathop{\sum}_{i=1,2}
     \big( \alpha_{i,n}^+ e^{\lambda_{i,n}^+t}E_{i,n}^+ +\alpha_{i,n}^- e^{\lambda_{i,n}^-t}E_{i,n}^- \big)  \\
     &\quad +\mathop{\sum}_{n\neq m} \big( \alpha_{2,n}^+ t e^{\lambda_{1,n}^+t} E_{1,n}^+ + \alpha_{2,n}^- te^{\lambda_{1,n}^-t} E_{1,n}^-\big)\\
     &\quad + \alpha_{2,m}^-  e^{\lambda_{1,m}t}  (t E_{2,m}^-+    \f{t^2}{2!} E_{1,m}^-+\f{t^3}{3!} E_{1,m}^+) \\
     &\quad + \alpha_{2,m}^+  e^{\lambda_{1,m}t}  (t E_{1,m}^-+\f{t^2}{2!} E_{1,m}^+)\\
     &\quad + \alpha_{1,m}^- t e^{\lambda_{1,m}t} E_{1,m}^+
\end{split}
\ee
Above all, we easily conclude by \e{con1},\e{sol1}, \e{sol2},\e{sol3},\e{sol4},\e{sol5},\e{sol6},\e{sol7} that
\be  \lab{decay1} E(t) =2\|U(t)\|^2_{\mathcal{H}} \leq M E(0) e^{-\omega t}, \quad \forall t\in [0, +\infty)
\ee
for some positive constants $M,\omega$ independent of the initial data.
This is the end of Proof of Proposition \ref{prop2}. \qed

From the above proof of Proposition \ref{prop2},  we get the corollary concerning the decay rate.
 \begin{cor}
 \lab{cor}
  i)   If  $B= \begin{pmatrix}
      a & b \\
      -b & a \\ \end{pmatrix} \ (a>0, b \in \mathbb{R}, b\neq 0 ) $;
      then one has \eqref{decay1}
      with the optimal decay rate  $\omega = -2\Re(\lambda_{1, 1}^+)=-2\Re(\lambda_{2, 1}^+) $.
  ii) If  $B= \begin{pmatrix}      a & 0 \\
      b & c \\ \end{pmatrix} \  (a>0, b\in \mathbb{R}, c>0) $,
      then one has
   \be  \label{decay2} E(t) =2\|U(t)\|^2_{\mathcal{H}} \leq M E(0) t^p e^{-\omega t}, \quad \forall t\in [0, +\infty).
    \ee
       with the optimal decay rate  $\omega = -2\max\{ \Re(\lambda_{1, 1}^+), \Re(\lambda_{2, 1}^+)\} $
      and
      \be p=
\begin{cases} 1,  \quad \text{\ if\quad} a=c=2 \text{\ and\ } b= 0\text{\quad or \ }a=2 \neq c \text{\quad or \ } a \neq 2 = c ;
   \\
   3,  \quad \text{\ if\quad } a=c=2 \text{\ and\ } b\neq 0;
   \\
   0, \quad \text{\ else.}
   \end{cases}
\ee
 \end{cor}


\subsection{Multiplier method}
In this subsection, we apply the multiplier method to establish the decay estimates of the total energy.
The key ingredient is to use an integral inequality (see Lemma \ref{lem5})  which leads to the exponential decay of the energy.
In this subsection, the initial data are assumed to be all real functions.

In Case i), System \e{wave} is reduced to
\be
 \lab{wave4}
  \left\{
 \begin{split}
  &u_{tt}- \D u+a u_{t}+b v_t =0 \quad\quad\  &\text{in } \ \Omega\times (0,+\infty),\\
 & v_{tt}- \D v-b u_t+a v_t=0\quad\quad\quad &\text{in } \ \Omega\times (0,+\infty),\\
  &u=v=0\quad\quad\quad\quad&\text{on} \ ~ \Gamma\times  (0,+\infty),\\
  &(u,u_t)(0)=(u^0,u^1),  \  ( v,v_t)(0)=(v^0,v^1)\quad\quad&\text{in} ~\  \Omega \qquad\qquad\quad
  \end{split}
 \right.
 \ee

 Using the multiplier $u_t$  to $u$-equation and   $v_t$ to $v$-equation, we obtain
 \beq \int _\Omega u_t(u_{tt}-\D u+a u_t+b v_t)+v_t(v_{tt}-\D v+a v_t-b u_t)dx=0
 \eeq
Integrating by parts and using the definition of the total energy $E(t)$, we get
  \be \lab{enre}
      \f{dE(t)}{dt}=-a\int_\Omega u_t^2+v_t^2dx\leq 0\ee
which indicates that  the energy decays if and only if $a>0$.
Then, we  can follow the proof of \cite[page 114, Theorem 8.13]{Komornik94} to conclude by  LaSalle Invariance Principle  \cite{LaSalle}
that System \ref{wave1} is asymptotically stable if $a>0$.

  Using the multiplier $u$  to $u$-equation and $v$ to $v$-equation,  we obtain
 for  $0\leq S\leq T$ that
  \beq      \int_S^T \int _\Omega u(u_{tt}-\D u+a u_t+b v_t)+v(v_{tt}-\D v-b u_t+a v_t)dxdt=0
      \eeq
and after integration by parts,
\beq \begin{split}
\int_\Omega uu_t+vv_tdx\bigg|^T_S-\int^T_S\int_\Omega u_t^2+v_t^2dxdt+\int^T_S\int_\Omega |\nabla u|^2+|\nabla v|^2dxdt \\
+\f{a}{2}\int_\Omega(u^2+v^2)dx\bigg|^T_S+b\int^T_S\int_\Omega uv_t-vu_t dxdt=0
\end{split}
\eeq
thus
\beq\begin{split}
\int^T_S\int_\Omega |\nabla u|^2+|\nabla v|^2dxdt
=\int^T_S\int_\Omega u_t^2+v_t^2dxdt-b\int^T_S\int_\Omega uv_t-vu_t dxdt   \\
-\int_\Omega uu_t+vv_t+\f{a}{2}(u^2+v^2)\bigg|^T_S dx
\end{split}
\eeq
Thanks  to  Cauchy Inequality and Poincar\'{e} Inequality,  we get for every $\varepsilon>0$
\beq \begin{split}
\int^T_S\int_\Omega |\nabla u|^2+|\nabla v|^2dxdt
&\leq (1+ C_{\varepsilon} ) \int^T_S\int_\Omega u_t^2+v_t^2dxdt+\varepsilon \int^T_S\int_\Omega u^2+v^2 dxdt \\
&\quad-\int_\Omega uu_t+vv_t+\f{a}{2}(u^2+v^2)\bigg|^T_S dx \\
&\leq  (1+ C_{\varepsilon} ) \int^T_S\int_\Omega u_t^2+v_t^2dxdt
+C\varepsilon \int^T_S\int_\Omega |\nabla u|^2+|\nabla v|^2dxdt   \\
&\quad+C(E(T)+E(S))
\end{split}
\eeq
Noting \e{enre} implies for $a>0$
  \beq
  \int_S^T\int_\Omega u_t^2+v_t^2 dxdt=\f{1}{a}(E(S)-E(T))\leq \f{1}{a}E(S)
  \eeq
Choosing $\varepsilon>0$ small we get
\beq \begin{split}
 \int^T_S\int_\Omega |\nabla u|^2+|\nabla v|^2dxdt\leq CE(S)
  \end{split}
\eeq
and further
\beq \lab{enes}
\int_S^T E(s)ds\leq C E(S),\quad \forall  0 \leq S \leq T
\eeq
Thanks to  Lemma \ref{lem5}, we conclude that  $E(t)$ decays to 0 exponentially.

In Case ii),  System \e{wave} is reduced to
\be
 \lab{wave2}
  \left\{
 \begin{split}
  &u_{tt}- \D u+a u_{t}=0 \quad\quad\  &\text{in } \ \Omega\times (0,+\infty),\\
 & v_{tt}- \D v+ bu_t+c v_t=0\quad\quad\quad &\text{in } \ \Omega\times (0,+\infty),\\
  &u=v=0\quad\quad\quad\quad&\text{on} \ ~ \Gamma\times  (0,+\infty),\\
  &(u,u_t)(0)=(u^0,u^1),  \  ( v,v_t)(0)=(v^0,v^1)\quad\quad&\text{in} ~\  \Omega \qquad\qquad\quad
  \end{split}
 \right.
 \ee

It is easy to see that if $a \leq 0$, then the energy of $u$ does not decay;  while if $c \leq 0$, by taking $u \equiv 0$, then
the energy of  $v$ does not decay.
Hence, the total energy $E(t)$ of \eqref{wave2} does not decay to 0 if $a \leq 0$ or  $c \leq0$.
It remains to prove that if $a>0$ and $c>0$, $E(t)$ decays to 0 exponentially.

For $\kappa>0$, we set the equivalent energy
 \be \lab{eqen}
   E_{\kappa}(t)=\int _{\Omega}\f{\kappa u_t^2}{2}+\f{v_t^2}{2}+\f{\kappa |\nabla u|^2}{2}+\f{|\nabla v|^2}{2}dx\ee
Using the multiplier $\kappa u_t$ to $u$-equation, $v_t$ to $v$-equation,
  \beq \f{dE_{\kappa}(t)}{dt}=-\int_\Omega \kappa a u_t^2+bu_tv_t+cv_t^2  dx \eeq
Choosing $\kappa>0$ suitably large, namely,
$b^2-4 \kappa a c<0$ then there exists $\delta>0$ such that
   \be  \lab{Ek} \f{dE_{\kappa}(t)}{dt}  \leq -\delta\int_\Omega \kappa u_t^2+v_t^2dx   \ee

Using the multiplier $\kappa u$ to $u$-equation, $v_t$ to $v$-equation,
\beq \begin{split}
\int_\Omega \kappa uu_t+vv_tdx\bigg|^T_S-\int^T_S\int_\Omega \kappa u_t^2+v_t^2dxdt
+\int^T_S\int_\Omega \kappa |\nabla u|^2+|\nabla v|^2dxdt
\\
+ \f{1}{2} \int_\Omega( \kappa a u^2+c v^2)dx\bigg|^T_S+b\int^T_S\int_\Omega vu_t dxdt=0
 \end{split}
\eeq
thus,
\beq  \begin{split}
 \int^T_S\int_\Omega \kappa |\nabla u|^2+|\nabla v|^2dxdt
 =\int^T_S\int_\Omega \kappa u_t^2+v_t^2dxdt-b\int^T_S\int_\Omega vu_t dxdt
 \\
 -\int_\Omega \kappa uu_t+vv_t+ \f{1}{2} ( \kappa a u^2+ c v^2)\bigg|^T_Sdx
 \end{split}
\eeq
Thanks  to  Cauchy Inequality and Poincar\'{e} Inequality,  we get for every $\varepsilon>0$
\beq \begin{split}
 \int^T_S\int_\Omega \kappa |\nabla u|^2+|\nabla v|^2dxdt
&\leq \int^T_S\int_\Omega \kappa u_t^2+v_t^2dxdt+  \int^T_S\int_\Omega \varepsilon v^2 +C_{\varepsilon}  \kappa u_t^2 dxdt
\\
&\quad -\int_\Omega \kappa uu_t+vv_t+ \f{1}{2} ( \kappa a u^2+ c v^2)\bigg|^T_Sdx   \\
&\leq  (1+C_\varepsilon) \int^T_S\int_\Omega \kappa u_t^2+v_t^2dxdt   \\
&\quad +C\varepsilon \int^T_S\int_\Omega  |\nabla v|^2dxdt
+C(E_{\kappa}(T)+E_{\kappa}(S))
 \end{split}
\eeq
Similarly as  as in Case i), we can prove by   \e{Ek} and choosing $\varepsilon >0$ small  the integral inequality for $E_{\kappa}$:
 \beq
\int_S^TE_{\kappa}(t)dt\leq CE_{\kappa}(S),\quad \forall 0\leq S\leq T.
\eeq
Then it follows by  Lemma \ref{lem5}  $E_{\kappa}(t)$ (or equivalently,  $E(t)$) decays to 0 exponentially.
This ends the proof of Proposition \ref{prop2}. \qed

\subsection{Frequency Domain Approach}

In this subsection, we use the frequency domain approach to prove  Proposition \ref{prop2}.
More precisely, we want to get the exponential stability of the semigroup through the uniform  estimate of the resolvent on the imaginary axis  by
Lemma \ref{lem6}  \cite{Huang85, Pruss84} (see also \cite{LiuZhengBook}).

For  $U_0=(u^0,u^1,v^0,v^1)\in \mathcal{H}$,  the solution of  \eqref{abseq} is
$U(t)=S(t)U_0$  where $\{S(t)\}_{t\geq0}$ is the associated $C_0$-semigroup of operator and $\|S(t)U_0\|^2_{\mathcal{H}} =2 E(t)$ for all $t \geq 0$.

In Case (i),
the energy relation \eqref{enre} yields that the energy decays, or equivalently $\{S(t)\}_{t\geq0}$ is a contraction, if and only if  $a>0$.
Next, we prove  Conditions \eqref{cond1} and  \eqref{cond2}  in Lemma \ref{lem6} are satisfied for $a>0$.

We start, by contradiction arguments, to assume that  \eqref{cond1}  does not hold, i.e.,
there exists $\xi \in \R $ and $U=(u,y,v,z)\in \mathcal{D}( \mathcal{A})$ with $\|U\|_{\mathcal{H}}=1$ such that
$(i\xi -\mathcal{A})U=0$, namely,
  \beq \left\{ \begin{split}
       &   \i\xi u-y = 0  \quad&\text{in} ~&\mathcal{H}_0^1 \\
         &  (\i\xi+a)y-\D u+b z= 0  \quad&\text{in} ~&\mathcal{L}^2\\
        &   \i\xi v-z= 0    \quad&\text{in}~ &\mathcal{H}_0^1\\
         &  (\i\xi +a)z-\D v-b y= 0 \quad&\text{in}~ &\mathcal{L}^2
        \end{split}
   \right.
  \eeq
By definition \eqref{inpro}, we easily calculate
   \be\lab{nei2}\Re \langle (\i\xi -\mathcal{A})U,U \rangle_{\mathcal{H}}
   =a(||y||^2_{\mathcal{L}^2}+||z||^2_{\mathcal{L}^2})=0
   \ee
Then it follows that
   $y=z=0 ~~\text{in}~\mathcal{L}^2$
and thus
\beq\left\{ \begin{array}{cc}
         \i\xi  u= 0 \quad\text{in}~ \mathcal{L}^2 \\
           \D u= 0 \quad\text{in} ~\mathcal{L}^2\\
         \i\xi v= 0  \quad \text{in}~ \mathcal{L}^2\\
           \D v= 0\quad\text{in}~ \mathcal{L}^2
        \end{array}
\right.\eeq
The theory of elliptic equation with Dirichlet boundary condition implies that
$u = v = 0 \in \mathcal{H}^2\bigcap \mathcal{H}^1_0$. Consequently we have  $U=0 \in \mathcal{H}$
which contradicts with $\|U\|_{\mathcal{H}}=1$. Therefore \eqref{cond1}  is satisfied for $a>0$.

Next we continue to prove that  \eqref{cond2}  holds for $a>0$. Otherwise, thanks to  the continuity of the resolvent
$\mathcal{R}(\i\xi,\mathcal{A})$ with respect to $\xi \in \R$,   \eqref{cond2}  is not valid at $\infty$, i.e.,
there exists $\{U_n=(u_n,y_n,v_n,z_n)\}_{n\in \Z^+}\subset \mathcal{D}( \mathcal{A})$ with  $\|U_n\|_{\mathcal{H}}=1$
and $\{ \xi_n\}_{n \in \Z^+} \subset \R$ with
$|\xi_n|\rightarrow+\infty, \|(\i\xi_n-\mathcal{A})U_n \|_{\mathcal{H}} \rightarrow 0 $ as $n\rightarrow +\infty$, namely,
 \beq\left\{ \begin{split}
         & \i\xi_n u_n-y_n \rightarrow 0~~& \text{in \ } &\mathcal{H}_0^1 \\
          & (\i\xi_n+a)y_n-\D u_n+b z_n\rightarrow 0 ~~&\text{in\ } &\mathcal{L}^2\\
          & \i\xi_n v_n-z_n\rightarrow 0  ~~ &\text{in\ } &\mathcal{H}_0^1\\
          & (\i\xi_n+a)z_n-\D v_n-b y_n\rightarrow 0~~&\text{in\ } &\mathcal{L}^2
        \end{split}
\right.\eeq
Similarly as \eqref{nei2},  we have
  \be \lab{nei3} \Re \langle (\i\xi- \mathcal{A})U_n,U_n \rangle_\mathcal{H}=a(||y_n||^2_{\mathcal{L}^2}+||z_n||^2_{\mathcal{L}^2})
   \ee
On the other hand,
 \beq
 \begin{split}
|\Re \langle (\i\xi_n -\mathcal{A})U_n,U_n \rangle_\mathcal{H}|\leq
||(  \i\xi_n-\mathcal{A})U_n||_\mathcal{H} \cdot ||U_n||_\mathcal{H}
=||( \i\xi_n-\mathcal{A})U_n||_\mathcal{H}\rightarrow 0 .
\end{split}
\eeq
Then
 $z_n\rightarrow 0 \in \mathcal{L}^2, y_n\rightarrow 0 \in  \mathcal{L}^2$, and thus
 \beq\left\{ \begin{split}
        \i\xi_n  u_n\rightarrow 0 ~ \quad&\text{in} &\mathcal{L}^2 \\
            \f{\D u_n}{\i \xi_n}\rightarrow 0~\quad&\text{in} &\mathcal{L}^2\\
       \i \xi_n  v_n\rightarrow 0  ~ \quad&\text{in} &\mathcal{L}^2\\
           \f{ \D v_n}{\i\xi_n}\rightarrow 0~\quad&\text{in} &\mathcal{L}^2
        \end{split}
  \right.\eeq
Thanks to  Nirenberg Inequality \cite[page 135, Theorem 5.2]{Sobolev},  we get
\beq
||\nabla u_n||_{\mathcal{L}^2}
\leq  C ||\D u_n||_{\mathcal{L}^2} ^\f{1}{2}||u_n||^\f{1}{2}_{\mathcal{L}^2}
\leq C ||\f{\D u_n}{\xi_n}||_{\mathcal{L}^2} ^\f{1}{2}||\xi_n u_n||^\f{1}{2}_{\mathcal{L}^2}\rightarrow 0
\eeq
and in a same way $\nabla v_n\rightarrow0 \in \mathcal{L}^2$.
Consequently  $||U_n||_{\mathcal{H}}\rightarrow 0$ as $n\rightarrow +\infty$.
This contradicts with $||U_n||_{\mathcal{H}}=1$ and implies  \eqref{cond2}  for $a>0$.
The exponential decay of the semigroup of operator $\{S(t)\} _{t \geq 0}$ associated to \eqref{abseq}  is a consequence of Lemma \ref{lem6}.


In Case (ii),
 it is easy to see, as in Section 3.2,  that if $a\leq 0$ or $c\leq 0$,  system \eqref{wave2} is not asymptotically stable.
It remains to prove that  \eqref{wave2}  is exponentially stable if $a>0$ and $c>0$.

We introduce an equivalent inner product of \eqref{inpro},
\be
\lab{inpro2}
\langle U, \tilde{U} \rangle_\mathcal{H_{\kappa}}=\int_\Omega \kappa  \nabla \overline{u} \cdot \nabla \tilde{u}
+\kappa \overline{y} \tilde{y}  + \nabla \overline{v} \cdot \nabla \tilde{v}  + \overline{z} \tilde{z}\, dx,
\ee
where $\kappa>0$. Clearly, $\mathcal{H}$ is a Hilbert space under $\langle \cdot, \cdot \rangle_{\mathcal{H_{\kappa}}}$ and is denoted by $\mathcal{H_{\kappa}}$ with the corresponding norm $||\cdot||_{\mathcal{H_{\kappa}}}$.
For  $U_0=(u^0,u^1,v^0,v^1)\in \mathcal{H_{\kappa}}$,  the solution of  \eqref{abseq} is
$U(t)=S(t)U_0$  where $\{S(t)\}_{t\geq0}$ is the associated $C_0$-semigroup of operator and
$\|S(t)U_0\|^2_{\mathcal{H_{\kappa}}} =2 E_{\kappa}(t)$ for all $t \geq 0$.
By Lemma \ref{lem6}, it suffices to prove that Conditions  \eqref{cond1}  and  \eqref{cond2}  hold for  $a>0$ and $c>0$.

We start, by contradiction arguments, to assume that  \eqref{cond1} does not hold, i.e.,
there exists $\xi \in \R $ and $U=(u,y,v,z)\in \mathcal{D}( \mathcal{A})$ with $\|U\|_{\mathcal{H_{\kappa}}}=1$ such that
$(\i\xi -\mathcal{A})U=0$, namely,
 \beq\left\{ \begin{split}
        &  \i\xi u-y = 0 \quad &\text{in} ~ &\mathcal{H}_0^1 \\
         &  (\i\xi+a)y-\Delta u = 0\quad&\text{in} ~ &\mathcal{L}^2\\
          & \i\xi v-z= 0 \quad &\text{in}~  &\mathcal{H}_0^1\\
         &  (\i\xi +c)z-\Delta v+b y= 0\quad&\text{in}~  &\mathcal{L}^2
        \end{split}
  \right.\eeq
We compute
  \be  \label{neiji3} \Re \langle (\i\xi -\mathcal{A})U,U \rangle _\mathcal{H_{\kappa}}
  =a \kappa ||y||^2_{\mathcal{L}^2}+c||z||^2_{\mathcal{L}^2}+b \Re \langle y,z \rangle_{\mathcal{L}^2}=0.
  \ee
  Taking $\kappa>0$ suitably large, namely,
$b^2-4 \kappa a c<0$ then there exists $\delta>0$ such that
 \be \label {delta1}0= \Re \langle (\i\xi -\mathcal{A})U,U \rangle _\mathcal{H_{\kappa}}
  \geq  \delta  ( ||y||^2_{\mathcal{L}^2}+ ||z||^2_{\mathcal{L}^2} ).
  \ee
Then   $y=z=0 \in \mathcal{L}^2$ and thus
 \beq \left\{ \begin{split}
          & u= 0 \quad&\text{in \ } &\mathcal{L}^2 \\
          & \Delta u= 0 \quad&\text{in \  } &\mathcal{L}^2\\
           & v= 0  \quad &\text{in \ } &\mathcal{L}^2\\
          & \Delta v= 0\quad&\text{in \ }& \mathcal{L}^2
        \end{split}
   \right. \eeq
 Obviously the theory of elliptic equation with Dirichlet boundary condition implies that
$u = v = 0 \in \mathcal{H}^2\bigcap \mathcal{H}^1_0$. Consequently we have  $U=0 \in \mathcal{H_{\kappa}}$
which contradicts with $\|U\|_{\mathcal{H_{\kappa}}}=1$ and yields that  \eqref{cond1}  is satisfied for $a>0$ and $c>0$.

Next we continue to prove that  \eqref{cond2}  holds for $a>0$. Otherwise, thanks to  the continuity of the resolvent
$\mathcal{R}(\i\xi,\mathcal{A})$ with respect to $\xi \in \R$,   \eqref{cond2}  is not valid at $\infty$, i.e.,
there exists $\{U_n=(u_n,y_n,v_n,z_n)\}_{n\in \Z^+}\subset \mathcal{D}( \mathcal{A})$ with  $\|U_n\|_{\mathcal{H_{\kappa}}}=1$
and $\{ \xi_n\}_{n \in \Z^+} \subset \R$ with
$|\xi_n|\rightarrow+\infty, \|(\i\xi_n-\mathcal{A})U_n \|_{\mathcal{H_{\kappa}}} \rightarrow 0 $ as $n\rightarrow +\infty$, namely,
  \beq\left\{ \begin{split}
         & \i\xi_n u_n-y_n \rightarrow 0  \quad &\text{in} ~&\mathcal{H}_0^1 \\
          & (\i\xi_n+a)y_n-\D u_n\rightarrow 0 \quad&\text{in} ~&\mathcal{L}^2\\
          & \i\xi_n v_n-z_n\rightarrow 0 \quad &\text{in} ~&\mathcal{H}_0^1\\
          & (\i\xi_n+c)z_n-\D v_n+b y_n\rightarrow 0\quad&\text{in}~&\mathcal{L}^2
        \end{split}
  \right.\eeq
Let  $ \kappa>\f{b^2}{4ac}$.
 We get  similarly as \eqref{neiji3} and \eqref{delta1} that
  \be \Re \langle (i\xi_n-A)U_n,U_n \rangle _\mathcal{H_{\kappa}}
  =a \kappa ||y_n||^2_{\mathcal{L}^2}+c||z_n||^2_{\mathcal{L}^2}+b \Re \langle y_n,z_n \rangle_{\mathcal{L}^2}
  \geq \delta (||y_n||^2_{\mathcal{L}^2}+||z_n||^2_{\mathcal{L}^2})
  \ee
 for some $\delta >0$.  On the other hand,
\beq
\begin{split}
|\Re \langle (\i\xi_n-\mathcal{A})U_n,U_n \rangle_\mathcal{H_{\kappa}}|
&\leq ||(\i\xi_n-\mathcal{A})U_n||_\mathcal{H_{\kappa}} \cdot ||V_n||_\mathcal{H_{\kappa}}   \\
&= ||(\i\xi_n-\mathcal{A})U_n||_\mathcal{H_{\kappa}} \rightarrow 0 .
\end{split}
\eeq
Hence we have
$y_n\rightarrow 0   \in \mathcal{L}^2, z_n\rightarrow 0  \in \mathcal{L}^2$ and thus
 \beq\left\{ \begin{split}
         &\i\xi_n  u_n\rightarrow 0 \quad &\text{in \ } &\mathcal{L}^2 \\
          &  \f{ \D u_n}{\i\xi_n} \rightarrow 0\quad&\text{in \ } &\mathcal{L}^2\\
          &\i\xi_n  v_n\rightarrow 0 \quad &\text{in \ }& \mathcal{L}^2\\
          &  \f{\D v_n}{\i\xi_n} \rightarrow 0\quad&\text{in \ } &\mathcal{L}^2
        \end{split}
 \right.\eeq
 Thanks to  Nirenberg Inequality \cite[page 135, Theorem 5.2]{Sobolev}, we get
\beq
||\nabla u_n||_{\mathcal{L}^2}
\leq  C ||\D u_n||_{\mathcal{L}^2} ^\f{1}{2}||u_n||^\f{1}{2}_{\mathcal{L}^2}
\leq C ||\f{\D u_n}{\xi_n}||_{\mathcal{L}^2} ^\f{1}{2}||\xi_n u_n||^\f{1}{2}_{\mathcal{L}^2}\rightarrow 0
\eeq
and in a same way $\nabla v_n\rightarrow0 \in \mathcal{L}^2$.
Consequently  $||U_n||_{\mathcal{H_{\kappa}}}\rightarrow 0$ as $n\rightarrow +\infty$.
This contradicts with $||U_n||_{\mathcal{H_{\kappa}}}=1$ and implies  \eqref{cond2}  for $a>0$ and $c>0$.
The exponential decay of the semigroup $\{S(t)\} _{t \geq 0}$ associated to \eqref{abseq}  is again a consequence of Lemma \ref{lem6}.

\section{Extension and Remarks}

\begin{thm}
If the boundary conditions in \eqref{wave} are replaced by mixed types of boundary conditions,  we can adopt the
above three approaches to conclude that \eqref{2} is still the sufficient and necessary condition of exponential stability of  the system.  For instance,
   \beq
 \lab{wave3}
  \left\{
 \begin{split}
  &u_{tt}- \D u+\alpha u_{t}+\beta v_{t} =0 \quad\quad\  &\text{in} \ &\Omega\times (0,+\infty),\\
 & v_{tt}- \D v+\gamma u_{t}+\eta v_{t}=0\quad\quad\quad &\text{in} \ &\Omega\times (0,+\infty),\\
 & u=v=0\quad\quad\quad\quad&\text{on} \  &\Gamma_0\times  (0,+\infty),\\
 & \f{\pa u}{\pa \nu }=     \f{\pa v}{\pa \nu }=0 \quad\quad\quad\quad&\text{on}\  &\Gamma_1\times  (0,+\infty),\\
 & (u,u_{t})(0)=(u^0,u^1),  \  ( v,v_{t})(0)=(v^0,v^1)\quad\quad&\text{in} \  &\Omega
  \end{split}
 \right.
 \eeq
 where  $\Gamma_0 \bigcup\Gamma_1=\pa \Omega,\Gamma_0\bigcap\Gamma_1=\emptyset $.
\end{thm}

\begin{thm} If the damping and the coupling are both acted on the boundary, we can adopt the
multiplier method to conclude that \eqref{2} is still the sufficient and necessary condition of exponential stability of  the following system
 \beq
 \lab{wave4}
  \left\{
 \begin{split}
  &u_{tt}- \D u =0 \quad\quad\  &\text{in } \ &\Omega\times (0,+\infty),\\
  &v_{tt}- \D v=0\quad\quad\quad &\text{in } \ &\Omega\times (0,+\infty),\\
 & u=v=0\quad\quad\quad\quad&\text{on} \  &\Gamma_0\times  (0,+\infty),\\
  &\f{\pa u}{\pa \nu }+\alpha u_t+\beta v_t=0     &\text{on}  \ &\Gamma_1\times  (0,+\infty),\\
   &\f{\pa v}{\pa \nu }+\gamma u_t+\eta v_t=0       &\text{on}    \  &\Gamma_1\times  (0,+\infty),\\
  &(u,u_{t})(0)=(u^0,u^1),  \  ( v,v_{t})(0)=(v^0,v^1)\quad\quad&\text{in} \  &\Omega
  \end{split}
 \right.
 \eeq
where $\Gamma_0 \bigcup\Gamma_1=\pa \Omega,\, \Gamma_0\bigcap\Gamma_1=\emptyset $ with some suitable geometric conditions (see \cite{1992-BLR, Komornik94}).
For instance,  one can assume that there exists  $x_0\in \R^n$, such that
  $  \Gamma_0 =\{x \in \pa \Omega | (x-x_0 )\cdot \nu (x)\leq 0\} $ and $  \Gamma_1 =\{x \in \pa \Omega | (x-x_0 )\cdot \nu (x) > 0\}  \neq \emptyset  $
where $\nu(x)$ is the unit outer normal filed.
\end{thm}

\begin{thm}
If we consider a wave system with multiple propagation speeds
\beq
 \lab{wave6}
  \left\{
 \begin{split}
 & u_{tt}- \D u+\alpha u_{t}+\beta v_{t} =0 \quad\quad\  &\text{in } \ &\Omega\times (0,+\infty),\\
  &v_{tt}-c^2 \D v+\gamma u_{t}+\eta v_{t}=0\quad\quad\quad& \text{in } \ &\Omega\times (0,+\infty),\\
 & u=v=0\quad\quad\quad\quad&\text{on} \  &\Gamma\times  (0,+\infty),\\
 & (u,u_{t})(0)=(u^0,u^1),  \  ( v,v_{t})(0)=(v^0,v^1)\quad\quad&\text{in} \  &\Omega
  \end{split}
 \right.
 \eeq
with $c \neq 1$, we can adopt the multiplier method or the frequency domain approach to prove the exponential stability of  the above system under the assumption
\beq
   \begin{cases}
     \alpha>0,\\ \eta>0    ,\\     \alpha\eta-\beta\gamma>0 \end{cases}\eeq
 \end{thm}

\begin{rem}[Comparison of Three Approaches]
The advantage of Riesz basis approach is to give the explicit expression of the solution
so that the relation between the exponential stability of the system and the coefficients can be relatively easy to discover.
In addition, the optimal  decay rate can be  obtained through very careful analysis of both the spectrum and the eigenvectors.
However, the Riesz basis properties of the eigenvectors are hard to check in higher dimensional case.
The multiplier method is rather simple to apply without restrains in space dimension and works for variable coefficients case (even for nonlinear problems),
but it requires strong geometric assumptions.
The frequency domain approach is also applicable for all space dimension and  for the variable coefficients case
 without much information of the eigenvalues and eigenvectors.
 Nevertheless, the optimal decay rate can not be obtained in general by multiplier method or by frequency domain approach.
\end{rem}

\appendix
\section{Appendix}
\begin{lem}
\lab{Riesz}
Let $\{e_n\}_{n\in \Z^+}$ be a Riesz basis on Hilbert space $\mathcal{H}$. Let $\{E_n\}_{n\in \Z^+}$  be a basis on $\mathcal{H}$.
If $\sum\limits_{n\in \Z^+} ||E_n-e_n||^2  < +\infty$, then $\{E_n\}_{n\in \Z^+}$ is a Riesz basis on $\mathcal{H}$.
\end{lem}

Lemma \ref{Riesz} is an equivalent form of Bari's Theorem, see \cite[page 317, Theorem 2.3] {Gohberg}.
The special case that
$\{e_n\}_{n\in \Z^+}$ is an orthogonal Riesz basis of $\mathcal{H}$ can be found as \cite[App. D, Theorem 3]{PoschelBook}.

\begin{lem}  \cite[page 103, Theorem 8.1]{Komornik94}
\lab{lem5}
Let $E:  [0,+\infty) \rightarrow [0,+\infty)$ be a non increasing function
Suppose that there exists $C>0$ such that
\beq
\int_S^T E(s)ds\leq C E(S),\quad \forall  0 \leq S \leq T,
\eeq
then
\beq
\lab{es}
E(t)\leq E(0)e^{1-\f{t}{C}}, \quad \forall t\geq C
\eeq
\end{lem}

\begin{lem} \cite[page 4, Theorem 1.3.2]{LiuZhengBook}
\lab{lem6}
Let
$\{S(t)=e^{t \mathcal{A}} \}_{t\geq 0}$ be a $C_0$-semigroup of contractions on a  Hilbert space $\mathcal{H}$.
 Then  $\{S(t)\}_{t\geq0}$ is exponentially stable if and only   the following conditions are satisfied
 \begin{align}  \label{cond1}
& (1) \  \i \R =\{ \i\xi \, | \,  \xi \in \R\}\subset \rho( \mathcal{A});   \\ \label{cond2}
 & (2)   \   \limsup_{ |\xi| \rightarrow +\infty} ||\mathcal{R}(\i\xi, \mathcal{A} )||<+\infty.
 \end{align}
\end{lem}

\end{CJK*}


\begin{thebibliography}{10}

\bibitem{Sobolev}
Robert~A. Adams and John J.~F. Fournier.
\newblock {\em Sobolev spaces}, volume 140 of {\em Pure and Applied Mathematics
  (Amsterdam)}.
\newblock Elsevier/Academic Press, Amsterdam, second edition, 2003.

\bibitem{2002-ACK}
Fatiha Alabau, Piermarco Cannarsa, and Vilmos Komornik.
\newblock Indirect internal stabilization of weakly coupled evolution
  equations.
\newblock {\em J. Evol. Equ.}, 2(2):127--150, 2002.

\bibitem{2002-AB-SICON}
Fatiha Alabau-Boussouira.
\newblock Indirect boundary stabilization of weakly coupled hyperbolic systems.
\newblock {\em SIAM J. Control Optim.}, 41(2):511--541 (electronic), 2002.

\bibitem{2010-AB-JDE}
Fatiha Alabau-Boussouira.
\newblock A unified approach via convexity for optimal energy decay rates of
  finite and infinite dimensional vibrating damped systems with applications to
  semi-discretized vibrating damped systems.
\newblock {\em J. Differential Equations}, 248(6):1473--1517, 2010.

\bibitem{2015-AWY}
Fatiha Alabau-Boussouira, Zhiqiang Wang, and Lixin Yu.
\newblock A one-step optimal energy decay formula for indirectly nonlinearly
  damped hyperbolic systems coupled by velocities.
\newblock {\em preprint, arXiv:1503.04126}, 2015.

\bibitem{1992-BLR}
Claude Bardos, Gilles Lebeau, and Jeffrey Rauch.
\newblock Sharp sufficient conditions for the observation, control, and
  stabilization of waves from the boundary.
\newblock {\em SIAM J. Control Optim.}, 30(5):1024--1065, 1992.

\bibitem{1994-CoxZZ}
Steven Cox and Enrique Zuazua.
\newblock The rate at which energy decays in a damped string.
\newblock {\em Comm. Partial Differential Equations}, 19(1-2):213--243, 1994.

\bibitem{Gohberg}
Israel Gohberg and Mark~G. Kre\u{i}n.
\newblock {\em Introduction to the theory of linear nonselfadjoint operators}.
\newblock Translated from the Russian by A. Feinstein. Translations of
  Mathematical Monographs, Vol. 18. American Mathematical Society, Providence,
  R.I., 1969.

\bibitem{GolubBook}
Gene~H. Golub and Charles~F. Van~Loan.
\newblock {\em Matrix computations}.
\newblock Johns Hopkins Studies in the Mathematical Sciences. Johns Hopkins
  University Press, Baltimore, MD, third edition, 1996.

\bibitem{Hzy}
Zhiyuan Hu.
\newblock Asymptotic synchronization for a coupled system of wave eqution (in
  chinese).
\newblock {\em Master Thesis}, 2014.

\bibitem{Huang85}
Falun Huang.
\newblock Characterization condition for exponential stability of linear
  dynamical systems in hilbert spaces.
\newblock {\em Trans. Amer. Math. Soc.}, 1(1):43--56, 1985.

\bibitem{Komornik94}
Vilmos Komornik.
\newblock {\em Exact controllability and stabilization: The Multiplier Method},
  volume~36 of {\em Paris-Chicester}.
\newblock Masson-John Wiley, 1994.

\bibitem{LaSalle}
Joseph~Pierre LaSalle.
\newblock Some extensions of {L}iapunov's second method.
\newblock {\em IRE Trans.}, CT-7:520--527, 1960.

\bibitem{2013-LR-CAM}
Tatsien Li and Bopeng Rao.
\newblock Exact synchronization for a coupled system of wave equations with
  {D}irichlet boundary controls.
\newblock {\em Chin. Ann. Math. Ser. B}, 34(1):139--160, 2013.

\bibitem{2014-LRH-COCV}
Tatsien Li, Bopeng Rao, and Long Hu.
\newblock Exact boundary synchronization for a coupled system of 1-{D} wave
  equations.
\newblock {\em ESAIM Control Optim. Calc. Var.}, 20(2):339--361, 2014.

\bibitem{2005-LiuRao-ZAMP}
Zhuangyi Liu and Bopeng Rao.
\newblock Characterization of polynomial decay rate for the solution of linear
  evolution equation.
\newblock {\em Z. Angew. Math. Phys.}, 56(4):630--644, 2005.

\bibitem{2007-LiuRao-JMAA}
Zhuangyi Liu and Bopeng Rao.
\newblock Frequency domain approach for the polynomial stability of a system of
  partially damped wave equations.
\newblock {\em J. Math. Anal. Appl.}, 335(2):860--881, 2007.

\bibitem{2009-LiuRao-DCDS}
Zhuangyi Liu and Bopeng Rao.
\newblock A spectral approach to the indirect boundary control of a system of
  weakly coupled wave equations.
\newblock {\em Discrete Contin. Dyn. Syst.}, 23(1-2):399--414, 2009.

\bibitem{LiuZhengBook}
Zhuangyi Liu and Songmu Zheng.
\newblock {\em Semigroups associated with dissipative systems}, volume 398 of
  {\em Chapman \& Hall/CRC Research Notes in Mathematics}.
\newblock Chapman \& Hall/CRC, Boca Raton, FL, 1999.

\bibitem{Pazy}
Amnon Pazy.
\newblock {\em Semigroups of linear operators and applications to partial
  differential equations}, volume~44 of {\em Applied Mathematical Sciences}.
\newblock Springer-Verlag, New York, 1983.

\bibitem{PoschelBook}
J{\"u}rgen P{\"o}schel and Eugene Trubowitz.
\newblock {\em Inverse spectral theory}, volume 130 of {\em Pure and Applied
  Mathematics}.
\newblock Academic Press, Inc., Boston, MA, 1987.

\bibitem{Pruss84}
Jan Pr{\"u}ss.
\newblock On the spectrum of {$C_{0}$}-semigroups.
\newblock {\em Trans. Amer. Math. Soc.}, 284(2):847--857, 1984.

\end{thebibliography}
\end{document}